\documentclass[a4paper,12pt]{article}
\usepackage[cp1251]{inputenc}
\usepackage[english]{babel}
\usepackage{amsmath,amssymb,euscript,amsthm}

\theoremstyle{plain}
\newtheorem{thm}{Theorem}
\newtheorem{lem}{Lemma}
\newtheorem{prop}{Proposition}
\theoremstyle{definition}
\newtheorem{nsl}{Corollary}

\begin{document}
UDC 519.24
\vskip10pt
\begin{center}
{\large \textbf{Equilibrium in Wright-Fisher models of population genetics}}
\vskip10pt
{\large \textbf{D.Koroliouk$^{\dag}$, V.S.Koroliuk$^*$}}\\
$^{\dag}$Institute of Telecommunications and Global Information Space \\
$^*$Institute of Mathematics \\
$^{\dag *}$Ukrainian Academy of Sciences, Kiev, Ukraine
\end{center}
\vskip10pt

\textbf{Abstract.}
For multivariant Wright-Fisher models in population genetics we introduce equilibrium states, expressed by fluctuations of probability relations, in distinction of the traditionally used fluctuations, expressed by the difference between the current value of a random process and its equilibrium value. \\
Then the drift component of the gene frequencies dynamic process, primarily espressed as a ratio of two quadratic forms, is transformed in a cubic parabola with a certain normalization factor.

\vskip10pt


\section{Introduction}

The population genetics models by Wright-Fisher are defined by regression functions which are determined by a ratio of two quadratic forms \cite[Ch.10]{Eth-Kur}.

However, equilibrium state, defined by equilibrium point of regression function, requires additional analysis (see, for ex., \cite{VK-DK, Sk-Hop-Sal}).

At the same time, the equilibrium is easily determined for the incremental regression function at each stage \cite{DKor5, DKor6, DK_8}.

In the present work, the models of population genetics of genotypes interaction are determined by difference evolution equations with regression functions of increments for the frequency probabilities of genotypes.

In this case, the equilibrium state of the probabilities frequency is given by the equilibrium of the regression function of increments, which is postulated by the form of such a function.

\section{Regression function of increments}

The probabilities of genotype frequencies at each stage $k\geq0$ are determined by the evolutionary process
$P(k)=(P_m(k), \ 0\leq m\leq M)$  with $M+1$ ($M\geq1$) finite number of the state set $E=\{e_0,e_1,\dots e_M\}$.

The dynamics of the frequency probabilities at the next  $k+1$-th  stage ($k\geq0$) is given by the regression function \cite[Ch.10]{Eth-Kur}
\begin{align}
&P_m (k+1):=W_m (p) / W(p) \ , \ \ 0\leq m\leq M \ , \ \ k\geq0,\\
&W_m(p):=p_m\sum_{n=0}^{M}W_{mn}p_n \ , \ \ 0\leq m\leq M,\\
&W(p):=\sum_{n=0}^{M}W_m(p).
\end{align}
The probabilities of frequencies obey the usual restrictions $0\leq p_m\leq 1$, $\sum_{n=0}^{M}p_n=1$. The respective restrictions for the survival parameters are   $0\leq W_{mn}\leq 1$, $0\leq m,n\leq M$.


The increment of probability at each stage
\begin{equation}
\Delta P_m(k+1):=P_m(k+1)-P_m(k) \ , \ \ 0\leq m\leq M \ , \ \ k\geq0,
\end{equation}
is given by the incremental regression function
\begin{equation}
\Delta P_m(k+1)=W_0^{(m)}(p) \ , \ \ 0\leq m\leq M,
\end{equation}
\begin{equation}
W_0^{(m)}(p)=V_0^{(m)}(p)/W(p) \ , \ \ V_0^{(m)}(p):=W_m(p)-p_mW(p) \ , \ \ 0\leq m\leq M.
\end{equation}
Let us introduce new parameters of survival:
\begin{equation}
V_{mn}:=1-W_{mn} \ , \ \ 0\leq m,n\leq M.
\end{equation}
Then the numerator of incremental regression function (6) is transformed to the form:
\begin{equation}
V_0^{(m)}(p)=p_m\left[\sum_{n=0}^{M}p_n(V_n,p)-(V_m,p)\right] \ , \ \ 0\leq m\leq M,
\end{equation}
and the normalizing denominator (3) has the form:
\begin{equation}
W(p)=1-\sum_{n=0}^{M}p_n(V_n,p).
\end{equation}
where the scalar product
\begin{equation}
(V_m,p):=\sum_{n=0}^{M}V_{mn}p_n \ , \ \ 0\leq m\leq M.
\end{equation}
Introduce the equilibriums of incremental regression functions (8) by the relations:
\begin{equation}
(V_m,\rho)=\pi \ , \ \ 0\leq m\leq M \ , \ \ \pi:=\prod_{n=0}^{M}\rho_n.
\end{equation}
The normalized constant $\pi$ is also generated by equilibriums $\rho=(\rho_m$, $0\leq m\leq M)$.
\begin{lem}
The equilibriums of incremental regression functions (6) - (10) are given by the relation:
\begin{equation}
\rho_m=\pi\overline{V}_m \ , \ \ 0\leq m\leq M \ , \ \ \pi:=\prod_{n=0}^{M}\rho_n,
\end{equation}
where $\overline{V}_m:=\sum_{n=0}^{M}\overline{V}_{mn}$, $0\leq m\leq M$ with the summands which are the elements of inverse matrix $\mathbb{V}^{-1}:=[\overline{V}_{mn}$, $0\leq m,n\leq M]$ with respect to the directing parameters matrix $\mathbb{V}=[V_{mn}$; $0\leq m,n\leq M]$, under the additional normalization condition
$\sum_{m=0}^{M}\overline{V}_m=\sum_{m,n=0}^{M}\overline{V}_{mn}=\pi^{-1}$.
\end{lem}
\begin{proof}[{\it Proof}]
The relation (11) means that
\begin{equation*}
\mathbb{V}\rho=\pi\textbf{1} \ , \ \ \pi\textbf{1}:=(\pi \ , \ \ 0\leq n\leq M).
\end{equation*}
Hence the vector of equilibriums has the following representation:
\begin{equation}
\rho=\pi\mathbb{V}^{-1}\textbf{1} \ , \ \ \rho_m=\overline{V}_{m}\pi \ , \ \ 0\leq m\leq M,
\end{equation}
that is, the assertion of Lemma (12).
\end{proof}

\begin{nsl}
The equilibria (12) provide the equilibrium state of the probability frequency (1):
\begin{equation}
V_0^{(m)}(\rho)\equiv0 \ , \ \ 0\leq m\leq M.
\end{equation}
\end{nsl}
\begin{nsl}
The equilibria (12) generate a representation of the scalar products (10) by \textit{fluctuations of the probability relations}:
\begin{equation}
(V_{m},p)=\pi p_m/\rho_m \ , \ \ 0\leq m\leq M.
\end{equation}
The normalizing constant $\pi$ is defined in (11).
\end{nsl}
First of all, note that relation (12) coincides with the definition of the equilibrium (11), under additional assumption that the directing parameters matrix $\mathbb{V}=V_m\delta_{mn}$; $0\leq m,n\leq M$, is diagonal. Hence we have the following
\begin{lem} There takes place the following relation:
\begin{equation}
\pi\rho^{-1}=\mathbb{V}\textbf{1},
\end{equation}
which coincides with formula (15).
\end{lem}
\vskip18pt
Now the incremental regression functions (6) - (9) with the relations (14) generate the following
\begin{prop}
The incremental regression functions with Wright-Fisher normalization is given by the relations:
\begin{align}
&W_0^{(m)}(p)=V_0^{(m)}(p)/W(p),\\
&V_0^{(m)}(p)=\pi\rho_m[\sum_{m=0}^{M}p_n^2/\rho_n-p_m/\rho_m],\\
&W(p)=1-\pi\sum_{m=0}^{M}p_n^2/\rho_n.
\end{align}
It is obvious the balance condition:
\begin{equation}
\sum_{m=0}^{M}V_0^{(m)}(p)=0,
\end{equation}
which in scalar form is the following:
\begin{equation}
\sum_{m=0}^{M} p_m \sum_{n=0}^{M} p_n^2/\rho_n - \sum_{m=0}^{M} p_m^2/\rho_m\equiv0.
\end{equation}
\end{prop}

\section{Equilibrium state}

The presence of equilibrium state is provided by equilibrium point of the incremental regression function:
\begin{equation}
V_0^{(m)}(\rho)=0 \ , \ \ 0\leq m\leq M \ , \ \ \rho=(\rho_m \ , \ \ 0\leq m\leq M).
\end{equation}
The normalizing Wright-Fisher factor has the form:
\begin{equation}
W(\rho)=1-\pi \ , \ \ \pi=\prod_{m=0}^{M}\rho_m.
\end{equation}
The equilibrium generated by the state $\rho=(\rho_m \ , \ \ 0\leq m\leq M)$, is interpreted by the convergence of evolutionary processes (1).
\begin{thm}
For any initial data: $0<P_m(0)<1$, $0\leq m\leq M$, evolutionary processes $P_m(k)$, $0\leq m\leq M$, $k\geq0$, which are determined by solutions of difference evolutionary equation (5) with the incremental regression function (16) - (18) converge, by $k\to\infty$,  to equilibrium
\begin{equation}
\lim_{k\to\infty}P_m(k)=\rho_m \ , \ \ 0\leq m\leq M.
\end{equation}
\end{thm}
\begin{proof}[\textit{Proof}] The property of the main components is used, which is specified by the sum
\begin{equation}
\sum_{n=0}^{M} p_n^2/\rho_n = \sum_{n=0}^{M} p_n(p_n/\rho_n),
\end{equation}
This means averaging the fluctuations of the probability relations $p_n/\rho_n$, $0\leq m\leq M$ on the distribution of frequencies at the current stage. In this case, the fluctuations of the ratios are equal to one for $p_n=\rho_n$,  $0\leq m\leq M$, and at the same time, the main component of the incremental regression function is also equal to one.\\
Consequently, the possible values of the frequency probabilities can be split into three zones:
\begin{equation*}
(+)\quad p_n<\rho_n\quad ; \qquad
(-)\quad p_n>\rho_n\quad ; \qquad
(0)\quad p_n=\rho_n\quad .
\end{equation*}
The signs of the incremental regression functions (16) - (18) in such a zones is the same:\\
in zone (+) the probabilities increase, in zone (-) they are decrease.

Therefore, there exists a limit (23) whose value is ensured by the necessary condition for the existence of a limit:
\begin{equation}
\lim_{k\to\infty}\Delta P_m(k+1)=0.
\end{equation}
\end{proof}

\section{Binary evolutionary process}

The binary EP $P_\pm(k)$, $k\geq0$, are determined by the following regression functions \cite{Eth-Kur}:
\begin{align}
&P_\pm(k+1)=W_\pm(p)/W(p) \ , \ \ k\geq0,\\
&W_\pm(p)=P_\pm(W_\pm p_\pm+p_\mp),\\
&W(0)=W_+(p_+)+W_-(p_-)=W_+p_+^2+2p_+p_-+W_-p_-^2.
\end{align}
The frequencies probabilities at $k$-th stage satisfy the usual conditions
$0\leq p_\pm\leq1$, $p_++p_-=1$. The survival parameters are also limited by the relation $0<W_\pm<1$.\\
For probability increments
\begin{equation}
\Delta P_\pm(k+1):=P_\pm(k+1)-P_\pm(k) \ , \ \ k\geq0,
\end{equation}
the corresponding regression functions of increments can be represented as follows:
\begin{equation}
W_0^\pm(p_\pm)=W_\pm(p_\pm)/W(p)-p_\pm,
\end{equation}
or equivalently
\begin{align}
&W_0^\pm(p_\pm)=V_0^\pm(p_\pm)/W(p),\\
&V_0^\pm(p_\pm)=W_\pm(p_\pm)-p_\pm W(p).
\end{align}
Now introduce direction parameters, based on the survival ones: $V_\pm:=1-W_\pm$ The relative equilibriums will be $\rho_\pm=V_\pm^{-1}$ with the normalization condition $V_++V_-=1$. Then the numerators (27) of the regression function has the following form:
\begin{equation}
W_\pm(p_\pm)=p_\pm(1-\pi p_\mp / \rho_\pm) \ , \ \ \pi:=\rho_+\rho_-.
\end{equation}
Therefore the numerator (32) transforms into the following:
\begin{equation}
V_0^\pm(p_\pm)=p_\mp W_\pm(p_\pm)-p_\pm W_\mp(p_\mp)=p_+p_-(\rho_\pm p_\mp-\rho_\mp p_\pm).
\end{equation}
The linear component has the following representation:
\begin{equation}
\rho_+p_--\rho_-p_+=-(p_+-\rho_+)=p_--\rho_-.
\end{equation}
So the regression functions of the increments of binary evolutionary processes are represented by the probability of fluctuations:
\begin{equation}
V_0^\pm(p_\pm)=- p_+p_-(p_\pm-\rho_\pm).
\end{equation}
The Wright-Fisher normalizing factor has the following representation:
\begin{equation}
W(p)=1-\pi[p_+^2/\rho_++p_-^2/\rho_-]=1-[\rho_-p_+^2+\rho_+p_-^2].
\end{equation}
The balance condition is also evident:
\begin{equation}
V_0^+(p_+)+V_0^-(p_-)\equiv0.
\end{equation}

\section{Conclusion}
The considered in the present work evolution processes serve as predictable components of stochastic models in population genetics and are represented by conditional mathematical expectations:
\begin{equation}\begin{split}
P_m(k+1):=E[S_N^{(m)}(k+1)\,|\,S_N(k)=P(k)] \ , \ \ 0\leq m\leq M \ , \ \ k\geq0.
\end{split}\end{equation}
The stochastic models in population genetics are determined by averaged sums
\begin{equation}
S_N(k):=\frac{1}{N}\sum^N_{n=1}\delta_n(k) \ , \ \ k\geq0.
\end{equation}
of random sample variables $\delta_n(k)$, $1\leq n\leq N$, which take values in a finite set with $M+1$ ($M\geq1$) states $E=\{e_0, e_1,\dots,e_M\}$ (see \cite{DKor10}). So the stochastic models (41) are defined by the sum of two components:
\begin{equation}
S_N(k+1)=V(S_N(k))+\Delta\mu_N(k+1) \ , \ \ k\geq0.
\end{equation}
The first, predictable component is generated by conditional mathematical expectations:
\begin{equation}
V_m(P_m(k))=P_m(k)+V_0^{(m)}(P_m(k))/W(P_m(k)) \ , \ \ 0\leq m\leq M \ , \ \ k\geq0.
\end{equation}
The second component forms a martingale differences
\begin{equation}
\Delta\mu_N(k+1)=S_N(k+1)-V(S_N(k)) \ , \ \ k\geq0,
\end{equation}
characterized by the first moments:
\begin{equation}\begin{split}
&E\Delta\mu_N^{(m)}(k+1)=0 \ , \ \ 0\leq m\leq M,\\
&E[(\Delta\mu_N^{(m)}(k+1))^2\,|\,S_N(k)]=\sigma_m^2(S_N(k)) \ , \ \ 0\leq m\leq M \ , \ \ k\geq0.
\end{split}\end{equation}
The conditional dispersion is determined by regression functions:
\begin{equation}
\sigma_m^2(p)=V_m(p)[1-V_m(p)] \ , \ \ 0\leq m\leq M.
\end{equation}
The asymptotical properties of stochastic models (42) - (45) by $N\to\infty$ as well as by $k\to\infty$, will be investigated in our next paper. The algorithms of phase merging \cite{Kor-Lim} and statistical estimation of drift parameter \cite{DK_13}, \cite{DK_lib} can be directly applied.

\end{document}